\documentclass[12pt, reqno]{amsart}
\usepackage{ amsmath,amsthm, amscd, amsfonts, amssymb, graphicx, color}
\usepackage[bookmarksnumbered, colorlinks, plainpages]{hyperref}
\newtheorem{thm}{Theorem}[section]
\newtheorem{cor}[thm]{Corollary}

\newtheorem{exam}[thm]{Example}
\numberwithin{equation}{section}

\begin{document}

\title{group inverse for anti-triangular block operator matrices}

\author{Huanyin Chen}
\author{Marjan Sheibani$^*$}
\address{
Department of Mathematics\\ Hangzhou Normal University\\ Hang -zhou, China}
\email{<huanyinchen@aliyun.com>}
\address{Women's University of Semnan (Farzanegan), Semnan, Iran}
\email{<sheibani@fgusem.ac.ir>}

 \thanks{$^*$Corresponding author}


\begin{abstract}
We present the existence of the group inverse and its representation for the block operator matrix $\left(
\begin{array}{cc}
E&I\\
F&0
\end{array}
\right)$ under the condition $FEF^{\pi}=0$. The group inverse for the anti-triangular block matrices with two identical subblocks under the same condition is thereby investigated. These extend the results of Zou, Chen and Mosi\'c (Studia Scient. Math. Hungar., 54(2017), 489--508), and Cao, Zhang and Ge (J. Appl. Math. Comput., 46(2014), 169--179).
\end{abstract}

\maketitle

\section{Introduction}

Let $\mathcal{B}(X)$ be a Banach algebra of all bounded linear operators over a Banach space $X$. An operator $T$ in $\mathcal{B}(X)$ has Drazin inverse provided that there exists some $S\in \mathcal{B}(X)$ such that $ST=TS, S=STS, T^n=T^{n+1}S$ for some $n\in {\Bbb N}$. Such $S$ is unique, if it exists, and we denote it by $T^D$.
The such smallest $n$ is called the Drazin index of $T$. If $T$ has Drazin index $1$, $T$ is said to have inverse $S$, and denote its group inverse by $T^{\#}$.
The group invertibility of the block operator matrices over a Banach space is attractive. It has interesting applications of resistance distances to the bipartiteness of graphs. Many authors have studied such problems from many different views, e.g., ~\cite{B,B2,C2,C1,C21}.

Let $E,F$ be bounded linear operators and $I$ be the identity operator over a Banach space $X$. It is attractive to investigate the Drazin (group) invertibility of the operator matrix $M=\left(
  \begin{array}{cc}
    E&I\\
    F&0
  \end{array}
\right)$. It was firstly posed by Campbell that the solutions to singular systems of differential equations is determined by the Drazin (group) invertibility of the preceding special matrix $M$.

Zou et al. studied the group inverse for $M$ under the condition $EF=0$. In Section 2, we present the existence of the group inverse and its representation for the block operator matrix $\left(
\begin{array}{cc}
E&I\\
F&0
\end{array}
\right)$ under the wider condition $FEF^{\pi}=0$.

In~\cite[Theorem 5]{C1}, Cao et al. considered the group inverse for a block matrix with identical subblocks over a right Ore domain. In Section 3, we further investigate the necessary and sufficient conditions for the existence and the representations of the the group inverse of a $2\times 2$ operator block matrix $$M=\left(
\begin{array}{cc}
E&F\\
F&0
\end{array}
\right)$$ with identical subblocks. The explicit formula of the group inverse of $M$ is also given under the same condition $FEF^{\pi}=0$.

Let $X$ be a Banach space. We use $\mathcal{B}(X)$ to denote the Banach algebra of bounded linear operator on $X$. If $T\in \mathcal{B}(X)$ has the Drazin inverse $T^D$, the element $T^{\pi}=I-TT^D$ is called the spectral idempotent of $T$. Let $p\in \mathcal{B}(X)$ be an idempotent operator, and let $T\in \mathcal{B}(X)$. Then we write $$T=pTp+pT(I-p)+(I-p)Tp+(I-p)T(I-p),$$
and induce a Pierce representation given by the operator matrix
$$T=\left(\begin{array}{cc}
pTp&pT(I-p)\\
(I-p)Tp&(I-p)T(I-p)
\end{array}
\right)_p.$$

\section{anti-triangular block matrices}

The aim of this section is to provide necessary and sufficient conditions on $E$ and $F$ so that the block operator matrix $\left(
\begin{array}{cc}
E&I\\
F&0
\end{array}
\right)$ has group inverse. We now derive

\begin{thm} Let $M=\left(
\begin{array}{cc}
E&I\\
F&0
\end{array}
\right)$ and $E,F,EF^{\pi}$ have Drazin inverses. If $FEF^{\pi}=0$, then the following are equivalent:
\end{thm}
\begin{enumerate}
\item [(1)]{\it $M$ has group inverse.} \vspace{-.5mm}
\item [(2)]{\it $F$ has group inverse and $E^{\pi}F^{\pi}=0$.}\\
\end{enumerate}
In this case, $$M^{\#}=\left(
\begin{array}{cc}
E^DF^{\pi}&F^{\#}+(E^DF^{\pi})^2-E^DF^{\pi}EF^{\#}\\
FF^{\#}&-FF^{\#}EF^{\#}
\end{array}
\right).$$
\begin{proof} Let $e=
\left(
\begin{array}{cc}
FF^D&0\\
0&I
\end{array}
\right)$. Then $M=
\left(
\begin{array}{cc}
a&b\\
c&d
\end{array}
\right)_e,$ where $$\begin{array}{c}
a=\left(
\begin{array}{cc}
FF^DE&FF^D\\
F^2F^D&0
\end{array}
\right),b=\left(
\begin{array}{cc}
0&0\\
FF^{\pi}&0
\end{array}
\right),\\
c=\left(
\begin{array}{cc}
F^{\pi}EFF^D&F^{\pi}\\
0&0
\end{array}
\right),d=\left(
\begin{array}{cc}
EF^{\pi}&0\\
0&0
\end{array}
\right).
\end{array}$$
Moreover, we compute that
$$a^{\#}=\left(
\begin{array}{cc}
0&F^D\\
FF^D&-FF^DEF^D
\end{array}
\right), d^D=\left(
\begin{array}{cc}
E^DF^{\pi}&0\\
0&0
\end{array}
\right).$$ Therefore we have $$a^{\pi}=\left(
\begin{array}{cc}
0&0\\
0&F^{\pi}
\end{array}
\right), d^{\pi}=\left(
\begin{array}{cc}
E^{\pi}F^{\pi}&0\\
0&0
\end{array}
\right).$$
$(2)\Rightarrow (1)$ Since $F$ has group inverse, we have $b=0$, and then
$$M=
\left(
\begin{array}{cc}
a&0\\
c&d
\end{array}
\right)_e.$$

Since $EF^{\pi}$ has group inverse, $d$ has group inverse and
and $d^{\#}=d^D$. In view of~\cite[Lemma 2.1]{C2},
$$M^D=\left(
\begin{array}{cc}
a^{\#}&0\\
z&d^{\#}
\end{array}
\right)_e,$$ where $$\begin{array}{lll}
z&=&(d^{\#})^2ca^{\pi}+d^{\pi}c(a^{\#})^2-d^{\#}ca^{\#}\\
&=&(d^{\#})^2ca^{\pi}-d^{\#}ca^{\#}\\
&=&\left(
\begin{array}{cc}
E^DF^{\pi}&0\\
0&0
\end{array}
\right)^2\left(
\begin{array}{cc}
F^{\pi}EFF^{\#}&F^{\pi}\\
0&0
\end{array}
\right)\left(
\begin{array}{cc}
0&0\\
0&F^{\pi}
\end{array}
\right)\\
&-&\left(
\begin{array}{cc}
E^DF^{\pi}&0\\
0&0
\end{array}
\right)\left(
\begin{array}{cc}
F^{\pi}EFF^{\#}&F^{\pi}\\
0&0
\end{array}
\right)\left(
\begin{array}{cc}
0&F^{\#}\\
FF^{\#}&-FF^{\#}EF^{\#}
\end{array}
\right)\\
&=&\left(
\begin{array}{cc}
0&(E^DF^{\pi})^2-E^DF^{\pi}EF^{\#}\\
0&0
\end{array}
\right).
\end{array}$$ Therefore $$M^D=\left(
\begin{array}{cc}
E^DF^{\pi}&F^{\#}+(E^DF^{\pi})^2-E^DF^{\pi}EF^{\#}\\
FF^{\#}&-FF^{\#}EF^{\#}
\end{array}
\right).$$ We check that $$d^{\pi}ca^{\pi}=\left(
\begin{array}{cc}
E^{\pi}F^{\pi}&0\\
0&0
\end{array}
\right)\left(
\begin{array}{cc}
F^{\pi}EFF^{\#}&F^{\pi}\\
0&0
\end{array}
\right)\left(
\begin{array}{cc}
0&0\\
0&F^{\pi}
\end{array}
\right)=0.$$  Then $M$ has group inverse.
Accordingly, $M^{\#}=M^D$, as required.

$(1)\Rightarrow (2)$ Write $M^{\#}=\left(
\begin{array}{cc}
X_{11}&X_{12}\\
X_{21}&X_{22}
\end{array}
\right)$. Then $MM^{\#}=M^{\#}M$, and so
$$\left(
\begin{array}{cc}
E&I\\
F&0
\end{array}
\right)\left(
\begin{array}{cc}
X_{11}&X_{12}\\
X_{21}&X_{22}
\end{array}
\right)=\left(
\begin{array}{cc}
X_{11}&X_{12}\\
X_{21}&X_{22}
\end{array}
\right)\left(
\begin{array}{cc}
E&I\\
F&0
\end{array}
\right).$$ Then we have
$$\begin{array}{c}
EX_{11}+X_{21}=X_{11}E+X_{12}F,\\
FX_{12}=X_{21}.
\end{array}$$
Since $MM^{\#}M=M$, we have
$$\begin{array}{c}
EX_{11}+X_{21}=I,\\
FX_{11}=0.
\end{array}$$
Therefore $$\begin{array}{lll}
F&=&(FE)X_{11}+FX_{21}\\
&=&(FE)F^dFX_{11}+FX_{21}\\
&=&(FEF^d)(FX_{11})+FX_{21}\\
&=&F^2X_{12}.
\end{array}$$ Hence $F$ has group inverse. Hence, $b=0$, and and so
$$M=\left(
\begin{array}{cc}
a&0\\
c&d
\end{array}
\right)_e$$ and $a$ has group inverse. Therefore $d$ has group inverse and $d^{\pi}ca^{\pi}=0$.

Since $d$ has group inverse, we see that $EF^{\pi}$ has group inverse. As $d^{\pi}ca^{\pi}=0$, we have
$$\left(
\begin{array}{cc}
E^{\pi}F^{\pi}&0\\
0&0
\end{array}
\right)\left(
\begin{array}{cc}
F^{\pi}EFF^D&F^{\pi}\\
0&0
\end{array}
\right)\left(
\begin{array}{cc}
0&0\\
0&F^{\pi}
\end{array}
\right)=\left(
\begin{array}{cc}
0&E^{\pi}F^{\pi}\\
0&0
\end{array}
\right)=0,$$ and therefore $E^{\pi}F^{\pi}=0$. This completes the proof.\end{proof}

\begin{cor} Let $M=\left(
\begin{array}{cc}
E&F\\
I&0
\end{array}
\right)$ and $E,F,EF^{\pi}$ have Drazin inverses. If $FEF^{\pi}=0$, then the following are equivalent:
\end{cor}
\begin{enumerate}
\item [(1)]{\it $M$ has group inverse.} \vspace{-.5mm}
\item [(2)]{\it $F$ has group inverse and $E^{\pi}F^{\pi}=0$.}\\
\end{enumerate}
In this case, $$M^{\#}=\left(
  \begin{array}{cc}
    \Gamma&\Delta\\
    \Lambda&\Xi\\
     \end{array}
\right),$$
where $$\begin{array}{lll}
\Gamma&=&F^{\pi}E^DF^{\pi},\\
\Delta&=&I-F^{\pi}E^DF^{\pi}E,\\
\Lambda&=&F^{\#}+(E^DF^{\pi})^2-E^DF^{\pi}EF^{\#},\\
\Xi&=&E^DF^{\pi}-F^{\#}E-(E^DF^{\pi})^2E+E^DF^{\pi}EF^{\#}E,
\end{array}$$
\begin{proof} Let $N=\left(
\begin{array}{cc}
E&I\\
F&0
\end{array}
\right)$. Then $$M=P^{-1}NP, P=\left(
\begin{array}{cc}
0&I\\
I&-E
\end{array}
\right).$$ Therefore $M$ has group inverse if and only if so does $N$, if and only if $F, EF^{\pi}$ have group inverse and $E^{\pi}F^{\pi}=0$, by Theorem 2.1.
In this case, $$\begin{array}{lll}
M^{\#}&=&P^{-1}N^{\#}P\\
 &=&\left(
\begin{array}{cc}
E&I\\
I&0
\end{array}
\right)N^{\#}\left(
\begin{array}{cc}
0&I\\
I&-E
\end{array}
\right)\\
&=&\left(
\begin{array}{cc}
I&F^{\pi}E^DF^{\pi}\\
E^DF^{\pi}&F^{\#}+(E^DF^{\pi})^2-E^DF^{\pi}EF^{\#}
\end{array}
\right)\left(
\begin{array}{cc}
0&I\\
I&-E
\end{array}
\right)\\
&=&\left(
  \begin{array}{cc}
    \Gamma&\Delta\\
    \Lambda&\Xi\\
     \end{array}
\right),
\end{array}$$
where $$\begin{array}{lll}
\Gamma&=&F^{\pi}E^DF^{\pi},\\
\Delta&=&I-F^{\pi}E^DF^{\pi}E,\\
\Lambda&=&F^{\#}+(E^DF^{\pi})^2-E^DF^{\pi}EF^{\#},\\
\Xi&=&E^DF^{\pi}-F^{\#}E-(E^DF^{\pi})^2E+E^DF^{\pi}EF^{\#}E,
\end{array}$$ as asserted.\end{proof}

We are now ready to prove the following.

\begin{thm} Let $M=\left(
\begin{array}{cc}
E&F\\
I&0
\end{array}
\right)$ and $E,F,EF^{\pi}$ have Drazin inverse. If $F^{\pi}EF=0$, then the following are equivalent:
\end{thm}
\begin{enumerate}
\item [(1)]{\it $M$ has group inverse.} \vspace{-.5mm}
\item [(2)]{\it $F$ has group inverse and $F^{\pi}E^{\pi}=0$.}\\
\end{enumerate}
In this case, $$M^{\#}=\left(
\begin{array}{cc}
F^{\pi}E^D&FF^{\#}\\
F^{\#}+(F^{\pi}E^D)^2-F^{\#}EF^{\pi}E^D&-F^{\#}EFF^{\#}
\end{array}
\right).$$
\begin{proof} We consider the transpose $M^T=\left(
\begin{array}{cc}
E^T&I\\
F^T&0
\end{array}
\right)$ of $M$. Then $M$ has group inverse if and only if so does $M^T$.
Applying Theorem 2.1, $M$ has group inverse if and only if $F^T, E^T(F^T)^{\pi}$ have group inverse and $(E^T)^{\pi}(F^T)^{\pi}=0$, i.e.,
$F, F^{\pi}E$ have group inverse and $F^{\pi}E^{\pi}=0$. In this case, we have
$$\begin{array}{l}
M^{\#}=[(M^T)^{\#}]^T=\\
{\small \left(
\begin{array}{cc}
(E^T)^D(F^T)^{\pi}&(F^T)^{\#}+((E^T)^D(F^T)^{\pi})^2-(E^T)^D(F^T)^{\pi}E^T(F^T)^{\#}\\
F^T(F^T)^{\#}&-F^T(F^T)^{\#}E^T(F^T)^{\#}
\end{array}
\right)^T},
\end{array}$$
as desired.\end{proof}

\begin{cor} Let $M=\left(
\begin{array}{cc}
E&I\\
F&0
\end{array}
\right)$ and $E,F$ have Drazin inverses. If $F^{\pi}EF=0$, then the following are equivalent:
\end{cor}
\begin{enumerate}
\item [(1)]{\it $M$ has group inverse.} \vspace{-.5mm}
\item [(2)]{\it $F$ has group inverse and $F^{\pi}E^{\pi}=0$.}\\
\end{enumerate}
In this case, $$M^{\#}=\left(
  \begin{array}{cc}
    \Gamma&\Delta\\
    \Lambda&\Xi\\
     \end{array}
\right),$$
where $$\begin{array}{lll}
\Gamma&=&F^{\pi}E^DF^{\pi},\\
\Delta&=&F^{\#}+(F^{\pi}E^D)^2-F^{\#}EF^{\pi}E^D,\\
\Lambda&=&I-EF^{\pi}E^DF^{\pi},\\
\Xi&=&F^{\pi}E^D-EF^{\#}-E(F^{\pi}E^D)^2+EF^{\#}EF^{\pi}E^D,
\end{array}$$
\begin{proof} Let $N=\left(
\begin{array}{cc}
E&F\\
I&0
\end{array}
\right)$. Then $$M=P^{-1}NP, P=\left(
\begin{array}{cc}
E&I\\
I&0
\end{array}
\right).$$ In view of Theorem 2.3, $$N^{\#}=\left(
\begin{array}{cc}
F^{\pi}E^D&FF^{\#}\\
F^{\#}+(F^{\pi}E^D)^2-F^{\#}EF^{\pi}E^D&-F^{\#}EFF^{\#}
\end{array}
\right).$$ Hence, $M$ has group inverse if and only if so does $N$, if and only if $F, F^{\pi}E$ have group inverse and $F^{\pi}E^{\pi}=0$, by Theorem 2.3.
Moreover, we have $$\begin{array}{lll}
M^{\#}&=&P^{-1}N^{\#}P\\
 &=&\left(
\begin{array}{cc}
0&I\\
I&-E
\end{array}
\right)N^{\#}\left(
\begin{array}{cc}
E&I\\
I&0
\end{array}
\right)\\
&=&\left(
  \begin{array}{cc}
    \Gamma&\Delta\\
    \Lambda&\Xi\\
     \end{array}
\right),
\end{array}$$
where $$\begin{array}{lll}
\Gamma&=&F^{\pi}E^DF^{\pi},\\
\Delta&=&F^{\#}+(F^{\pi}E^D)^2-F^{\#}EF^{\pi}E^D,\\
\Lambda&=&I-EF^{\pi}E^DF^{\pi},\\
\Xi&=&F^{\pi}E^D-EF^{\#}-E(F^{\pi}E^D)^2+EF^{\#}EF^{\pi}E^D,
\end{array}$$ as asserted. \end{proof}

We come now to prove:

\begin{cor} Let $M=\left(
\begin{array}{cc}
E&F\\
I&0
\end{array}
\right)$ and $E,F,EF^{\pi}$ have Drazin inverses. If $EF=\lambda FE (\lambda \in {\Bbb C})$ or $EF^2=FEF$, then the following are equivalent:
\end{cor}
\begin{enumerate}
\item [(1)]{\it $M$ has group inverse.} \vspace{-.5mm}
\item [(2)]{\it $F$ have group inverse and $F^{\pi}E^{\pi}=0$.}\\
\end{enumerate}
In this case, $$M^{\#}=\left(
\begin{array}{cc}
F^{\pi}E^D&FF^{\#}\\
F^{\#}+(F^{\pi}E^D)^2-F^{\#}EF^{\pi}E^D&-F^{\#}EFF^{\#}
\end{array}
\right).$$
\begin{proof} If $EF=\lambda FE (\lambda \in {\Bbb C})$, then $F^{\pi}EF=\lambda F^{\pi}FE=0$. If $EF^2=FEF$, then
$F^{\pi}EF=F^{\pi}EF^2F^{\#}=F^{\pi}FEFF^{\#}=0$. This completes the proof by Corollary 2.4.\end{proof}

\section{block matrices with identical subblocks}

In ~\cite{C1}, Cao et al. considered the group inverse for block matrices with identical subblocks over a right Ore domain. In this section we are concerned with
the group inverse for block operator matrices with identical subblocks over a Banach space.

\begin{thm} Let $M=\left(
\begin{array}{cc}
E&F\\
F&0
\end{array}
\right)$ and $E,EF^{\pi}$ have Drazin inverse and $F$ has group inverse. If $FEF^{\pi}=0$, then the following are equivalent:
\end{thm}
\begin{enumerate}
\item [(1)]{\it $M$ has group inverse.} \vspace{-.5mm}
\item [(2)]{\it $EE^{\pi}F^{\pi}=0$.}
\end{enumerate}
In this case, $$M^{\#}=\left(
  \begin{array}{cc}
    \Gamma&\Delta\\
    \Lambda&\Xi\\
     \end{array}
\right),$$
where $$\begin{array}{rll}
\Gamma&=&[I-E^{\pi}F^{\pi}][E^DF^{\pi}+E^{\pi}F^{\pi}E(F^{\#})^2]+E^{\pi}F^{\pi}E(F^{\#})^2,\\
\Delta&=&[I-E^{\pi}F^{\pi}][F^{\#}-E^{\pi}F^{\pi}E(F^{\#})^2EF^{\#}-E^DF^{\pi}EF^{\#}]\\
&-&E^{\pi}F^{\pi}E(F^{\#})^2EF^{\#},\\
\Lambda&=&F[E^DF^{\pi}+E^{\pi}F^{\pi}E(F^{\#})^2]^2+F^{\#}-FE^{\pi}F^{\pi}[E(F^{\#})^2]^2\\
&-&FE^DF^{\pi}E(F^{\#})^2,\\
\Xi&=&[FE^DF^{\pi}+FE^{\pi}F^{\pi}E(F^{\#})^2][F^{\#}-E^{\pi}F^{\pi}E(F^{\#})^2EF^{\#}\\
&-&E^DF^{\pi}EF^{\#}]-[F^{\#}-FE^{\pi}F^{\pi}E(F^{\#})^2E(F^{\#})^2\\
&-&FE^DF^{\pi}E(F^{\#})^2]EF^{\#}.
\end{array}$$
\begin{proof}
$(1)\Rightarrow (2)$ Obviously, we have
$$\begin{array}{c}
M=\left(
\begin{array}{cc}
F^{\pi}E&F\\
F&0
\end{array}
\right)\left(
\begin{array}{cc}
I&I\\
F^{\#}E&0
\end{array}
\right),\\
M^2=\left(
\begin{array}{cc}
EF^{\pi}E+F^2&EF\\
0&F^2
\end{array}
\right)\left(
\begin{array}{cc}
I&I\\
F^{\#}E&0
\end{array}
\right).
\end{array}$$
Write $M^{\#}=\left(
\begin{array}{cc}
X_{11}&X_{12}\\
X_{21}&X_{22}
\end{array}
\right)$. Then $M^{\#}M^2=M$, and so
$$\left(
\begin{array}{cc}
X_{11}&X_{12}\\
X_{21}&X_{22}
\end{array}
\right)\left(
\begin{array}{cc}
EF^{\pi}E+F^2&EF\\
0&F^2
\end{array}
\right)=\left(
\begin{array}{cc}
F^{\pi}E&F\\
F&0
\end{array}
\right).$$ Therefore $$X_{11}EF^{\pi}E+X_{11}F^2=F^{\pi}E,$$ hence,
$$X_{11}EF^{\pi}EF^{\pi}+X_{11}F^2F^{\pi}=F^{\pi}EF^{\pi}.$$ It follows that
$$X_{11}(EF^{\pi})^2=F^{\pi}EF^{\pi}=EF^{\pi}.$$ Hence $EF^{\pi}$ has group inverse. Then $(EF^{\pi})^{\#}=E^DF^{\pi}.$ Therefore
$$\begin{array}{lll}
EF^{\pi}&=&(EF^{\pi})^{\#}EF^{\pi}EF^{\pi}\\
&=&E^DF^{\pi}EF^{\pi}EF^{\pi}\\
&=&EE^DF^{\pi};
\end{array}$$ hence, $EE^{\pi}F^{\pi}=0$.

$(2)\Rightarrow (1)$ Since $EE^{\pi}F^{\pi}=0$, we have $E^DF^{\pi}(EF^{\pi})^2=EF^{\pi}$, and so $EF^{\pi}$ has group inverse.

Let $N=\left(
\begin{array}{cc}
E&I\\
F^2&0
\end{array}
\right)$. Choose $e=\left(
\begin{array}{cc}
FF^{\#}&0\\
0&I
\end{array}
\right).$ Then
$$a=\left(
\begin{array}{cc}
FF^{\#}E&FF^{\#}\\
F^2&0
\end{array}
\right),c=\left(
\begin{array}{cc}
F^{\pi}EFF^{\#}&F^{\pi}\\
0&0
\end{array}
\right),d=\left(
\begin{array}{cc}
EF^{\pi}&0\\
0&0
\end{array}
\right)$$ and $b=0$. Then
$$N=\left(
\begin{array}{cc}
a&0\\
c&d
\end{array}
\right)_e.$$
Moreover, we have $$a^{\#}=\left(
\begin{array}{cc}
0&(F^{\#})^2\\
FF^{\#}&-FF^{\#}E(F^{\#})^2
\end{array}
\right), d^{\#}=\left(
\begin{array}{cc}
E^DF^{\pi}&0\\
0&0
\end{array}
\right).$$ We compute that $$a^{\pi}=\left(
\begin{array}{cc}
0&0\\
0&F^{\pi}
\end{array}
\right), d^{\pi}=\left(
\begin{array}{cc}
E^{\pi}F^{\pi}&0\\
0&I
\end{array}
\right).$$ Obviously, $a^{\pi}cd^{\pi}=0$. So $N$ has group inverse.
Moreover, we have $$N^{\#}=\left(
\begin{array}{cc}
a^{\#}&0\\
z&d^{\#}
\end{array}
\right),$$ where $$z=d^{\pi}c(a^{\#})^2+(d^{\#})^2ca^{\pi}-d^{\#}ca^{\#}.$$
Clearly, $$\begin{array}{rll}
d^{\pi}c(a^{\#})^2&=&\left(
\begin{array}{cc}
E^{\pi}F^{\pi}E(F^{\#})^2&-E^{\pi}F^{\pi}E(F^{\#})^2E(F^{\#})^2\\
0&0
\end{array}
\right),\\
(d^{\#})^2ca^{\pi}&=&\left(
\begin{array}{cc}
0&E^DF^{\pi}E^DF^{\pi}\\
0&0
\end{array}
\right),\\
d^{\#}ca^{\#}&=&\left(
\begin{array}{cc}
0&E^DF^{\pi}E(F^{\#})^2\\
0&0
\end{array}
\right).
\end{array}$$
Hence we compute that $z=(z_{ij})$, where
$$\begin{array}{lll}
z_{11}&=&E^{\pi}F^{\pi}E(F^{\#})^2,\\
z_{12}&=&E^DF^{\pi}E^DF^{\pi}-E^{\pi}F^{\pi}E(F^{\#})^2E(F^{\#})^2-E^DF^{\pi}E(F^{\#})^2,\\
z_{21}&=&0, z_{22}=0.
\end{array}$$
Therefore $$N^{\#}=\left(
\begin{array}{cc}
\alpha&\beta\\
\gamma&\delta
\end{array}
\right),$$
where $$\begin{array}{lll}
\alpha&=&E^DF^{\pi}+E^{\pi}F^{\pi}E(F^{\#})^2,\\
\beta&=&(F^{\#})^2+E^DF^{\pi}E^DF^{\pi}-E^{\pi}F^{\pi}E(F^{\#})^2E(F^{\#})^2-E^DF^{\pi}E(F^{\#})^2,\\
\gamma&=&FF^{\#},\\
\delta&=&-FF^{\#}E(F^{\#})^2.
\end{array}$$
Hence, we have
$$\begin{array}{ll}
&NN^{\#}=N^{\#}N\\
=&\left(
\begin{array}{cc}
E&I\\
F^2&0
\end{array}
\right)\left(
\begin{array}{cc}
\alpha&\beta\\
\gamma&\delta
\end{array}
\right)\\
=&\left(
\begin{array}{cc}
\alpha&\beta\\
\gamma&\delta
\end{array}
\right)\left(
\begin{array}{cc}
E&I\\
F^2&0
\end{array}
\right)\\
=&\left(
\begin{array}{cc}
EE^DF^{\pi}+FF^{\#}&\alpha\\
F^2\alpha&FF^{\#}
\end{array}
\right).
\end{array}$$ Thus we have
$$N^{\pi}=\left(
\begin{array}{cc}
E^{\pi}F^{\pi}&-\alpha\\
-F^2\alpha&F^{\pi}
\end{array}
\right).$$
Clearly, we check that
$$\begin{array}{c}
M=\left(
\begin{array}{cc}
E&I\\
F&0
\end{array}
\right)\left(
\begin{array}{cc}
I&0\\
0&F
\end{array}
\right),\\
N=\left(
\begin{array}{cc}
I&0\\
0&F
\end{array}
\right)\left(
\begin{array}{cc}
E&I\\
F&0
\end{array}
\right).
\end{array}$$
By virtue of Cline's formula, $M$ has Drazin inverse.

We see that
$$\begin{array}{ll}
&\left(
\begin{array}{cc}
E&I\\
F&0
\end{array}
\right)N^{\pi}\left(
\begin{array}{cc}
I&0\\
0&F
\end{array}
\right)\\
=&\left(
\begin{array}{cc}
E&I\\
F&0
\end{array}
\right)\left(
\begin{array}{cc}
E^{\pi}F^{\pi}&-\alpha\\
-F^2\alpha&F^{\pi}
\end{array}
\right)\left(
\begin{array}{cc}
I&0\\
0&F
\end{array}
\right)\\
=&\left(
\begin{array}{cc}
-F^2\alpha&-E\alpha F\\
FE^{\pi}F^{\pi}&-F\alpha F
\end{array}
\right).
\end{array}$$ We compute that
$$\begin{array}{rll}
F^2\alpha&=&F^2E^DF^{\pi}+F^2E^{\pi}F^{\pi}E(F^{\#})^2\\
&=&F^2EF^{\pi}(E^DF^{\pi})^2+F^2(EF^{\pi})(E^DF^{\pi})E(F^{\#})^2=0,\\
E\alpha F&=&EE^{\pi}F^{\pi}EF^{\#}=0,\\
FE^{\pi}F^{\pi}&=&FE^DEF^{\pi}=F(E^DF^{\pi})(EF^{\pi})=F(EF^{\pi})(E^DF^{\pi})=0,\\
F\alpha F&=&FE^{\pi}F^{\pi}EF^{\#}=F(E^DF^{\pi})(EF^{\pi})EF^{\#}\\
&=&F(EF^{\pi})(E^DF^{\pi})EF^{\#}=0.
\end{array}$$ Hence $M=MM^DM$, i.e., $M$ has group inverse. Thus we have $M^{\#}=M^D$.

Moreover, we have $$\begin{array}{lll}
M^{D}&=&\left(
\begin{array}{cc}
E&I\\
F&0
\end{array}
\right)(N^{\#})^2\left(
\begin{array}{cc}
I&0\\
0&F
\end{array}
\right)\\
&=&\left(
\begin{array}{cc}
E&I\\
F&0
\end{array}
\right)\left(
\begin{array}{cc}
\alpha&\beta\\
\gamma&\delta
\end{array}
\right)^2\left(
\begin{array}{cc}
I&0\\
0&F
\end{array}
\right)\\
&=&\left(
  \begin{array}{cc}
    \Gamma&\Delta\\
    \Lambda&\Xi\\
     \end{array}
\right),
\end{array}$$
where $$\begin{array}{rll}
\Gamma&=&(E\alpha+\gamma)\alpha+(E\beta+\delta)\gamma,\\
\Delta&=&(E\alpha+\gamma)\beta F+(E\beta+\delta)\delta F,\\
\Lambda&=&F(\alpha^2+\beta\gamma),\\
\Xi&=&F(\alpha\beta +\beta\delta)F.
\end{array}$$ Therefore we complete the proof by the direct computation.\end{proof}

\begin{cor} Let $M=\left(
\begin{array}{cc}
E&F\\
F&0
\end{array}
\right)$ and $E,EF^{\pi}$ have Drazin inverse and $F$ has group inverse. If $F^{\pi}EF=0$, then the following are equivalent:\end{cor}
\begin{enumerate}
\item [(1)]{\it $M$ has group inverse.} \vspace{-.5mm}
\item [(2)]{\it $F^{\pi}E^{\pi}E=0$.}
\end{enumerate}
In this case, $$M^{\#}=\left(
  \begin{array}{cc}
    \Gamma&\Delta\\
    \Lambda&\Xi\\
     \end{array}
\right),$$
where $$\begin{array}{rll}
\Gamma&=&[F^{\pi}E^D+(F^{\#})^2EF^{\pi}E^{\pi}][I-F^{\pi}E^{\pi}]+(F^{\#})^2EF^{\pi}E^{\pi},\\
\Delta&=&[F^{\#}-F^{\#}E(F^{\#})^2EF^{\pi}E^{\pi}-F^{\#}EF^{\pi}E^D][I-F^{\pi}E^{\pi}]\\
&-&F^{\#}E(F^{\#})^2EF^{\pi}E^{\pi},\\
\Lambda&=&[F^{\pi}E^D+(F^{\#})^2EF^{\pi}E^{\pi}]^2F+F^{\#}-[(F^{\#})^2E]^2F^{\pi}E^{\pi}F\\
&-&(F^{\#})^2EF^{\pi}E^DF,\\
\Xi&=&[F^{\#}-F^{\#}E(F^{\#})^2EF^{\pi}E^{\pi}-F^{\#}EF^{\pi}E^D]\\
&&[F^{\pi}E^DF+(F^{\#})^2EF^{\pi}E^{\pi}F]-F^{\#}E[F^{\#}\\
&-&(F^{\#})^2E(F^{\#})^2EF^{\pi}E^{\pi}F-(F^{\#})^2EF^{\pi}E^DF].
\end{array}$$
\begin{proof} By virtue of Cline's formula, $F^{\pi}E$ has Drazin inverse. Then the proof is complete by applying Theorem 3.1 to the transpose $M^T=\left(
\begin{array}{cc}
E^T&F^T\\
F^T&0
\end{array}
\right).$\end{proof}

\begin{cor} Let $M=\left(
\begin{array}{cc}
E&F\\
F&0
\end{array}
\right)$ and $E,F$ have group inverse, $EF^{\pi}$ has Drazin inverse. If $F^{\pi}EF=0$, then $M$ has group inverse. In this case, $$M^{\#}=\left(
  \begin{array}{cc}
    \Gamma&\Delta\\
    \Lambda&\Xi\\
     \end{array}
\right),$$
where $$\begin{array}{rll}
\Gamma&=&[I-E^{\pi}F^{\pi}][E^{\#}F^{\pi}+E^{\pi}F^{\pi}E(F^{\#})^2]+E^{\pi}F^{\pi}E(F^{\#})^2,\\
\Delta&=&[I-E^{\pi}F^{\pi}][F^{\#}-E^{\pi}F^{\pi}E(F^{\#})^2EF^{\#}-E^{\#}F^{\pi}EF^{\#}]\\
&-&E^{\pi}F^{\pi}E(F^{\#})^2EF^{\#},\\
\Lambda&=&F[E^{\#}F^{\pi}+E^{\pi}F^{\pi}E(F^{\#})^2]^2+F^{\#}-FE^{\pi}F^{\pi}[E(F^{\#})^2]^2\\
&-&FE^{\#}F^{\pi}E(F^{\#})^2,\\
\Xi&=&[FE^{\#}F^{\pi}+FE^{\pi}F^{\pi}E(F^{\#})^2][F^{\#}-E^{\pi}F^{\pi}E(F^{\#})^2EF^{\#}\\
&-&E^{\#}F^{\pi}EF^{\#}]-[F^{\#}-FE^{\pi}F^{\pi}E(F^{\#})^2E(F^{\#})^2\\
&-&FE^{\#}F^{\pi}E(F^{\#})^2]EF^{\#}.
\end{array}$$
\end{cor}
\begin{proof} Since $E$ has group inverse, we see that $EE^{\pi}=0$, and so $EE^{\pi}F^{\pi}=0$. In light of Theorem 3.1, $M$ has group inverse. Therefore we obtain the representation of $M^{\#}$ by the formula in Theorem 3.1. \end{proof}

As an immediate consequence of Corollary 3.3, we have

\begin{cor} Let $M=\left(
\begin{array}{cc}
E&F\\
F&0
\end{array}
\right)$ and $E,F$ have group inverse, $EF^{\pi}$ has Drazin inverse. If $EF=\lambda FE ~(\lambda \in {\Bbb C})$ or $EF^2=FEF$, then $M$ has group inverse. In this case, $$M^{\#}=\left(
  \begin{array}{cc}
    \Gamma&\Delta\\
    \Lambda&\Xi\\
     \end{array}
\right),$$
where $$\begin{array}{rll}
\Gamma&=&[I-E^{\pi}F^{\pi}][E^{\#}F^{\pi}+E^{\pi}F^{\pi}E(F^{\#})^2]+E^{\pi}F^{\pi}E(F^{\#})^2,\\
\Delta&=&[I-E^{\pi}F^{\pi}][F^{\#}-E^{\pi}F^{\pi}E(F^{\#})^2EF^{\#}-E^{\#}F^{\pi}EF^{\#}]\\
&-&E^{\pi}F^{\pi}E(F^{\#})^2EF^{\#},\\
\Lambda&=&F[E^{\#}F^{\pi}+E^{\pi}F^{\pi}E(F^{\#})^2]^2+F^{\#}-FE^{\pi}F^{\pi}[E(F^{\#})^2]^2\\
&-&FE^{\#}F^{\pi}E(F^{\#})^2,\\
\Xi&=&[FE^{\#}F^{\pi}+FE^{\pi}F^{\pi}E(F^{\#})^2][F^{\#}-E^{\pi}F^{\pi}E(F^{\#})^2EF^{\#}\\
&-&E^{\#}F^{\pi}EF^{\#}]-[F^{\#}-FE^{\pi}F^{\pi}E(F^{\#})^2E(F^{\#})^2\\
&-&FE^{\#}F^{\pi}E(F^{\#})^2]EF^{\#}.
\end{array}$$
\end{cor}
\begin{proof} As in proof of Corollary 2.5, we obtain the result by Corollary 3.3.\end{proof}

\begin{exam} Let $M=\left(
\begin{array}{cc}
E&F\\
F&0
\end{array}
\right)$, where $E=\left(
\begin{array}{cc}
1& 2\\
0&-1
\end{array}
\right)$ and $F=\left(
\begin{array}{cc}
i& i\\
0&0
\end{array}
\right), i^2=-1$. Then $$M^{\#}=\left(
\begin{array}{cccc}
0&1&-i&-i\\
0&-1&0&0\\
-i&-i&1&1\\
0&0&0&0
\end{array}
\right).$$\end{exam}
\begin{proof} We see that $$\begin{array}{c}
E^{\#}=\left(
\begin{array}{cc}
1&2\\
0&-1
\end{array}
\right), E^{\pi}=0;\\
F^{\#}=\left(
\begin{array}{cc}
-i& -i\\
0&0
\end{array}
\right), F^{\pi}=\left(
\begin{array}{cc}
0&-1\\
0&1
\end{array}
\right).
\end{array}$$ Hence we check that $FEF^{\pi}=0, EE^{\pi}F^{\pi}=0.$ Construct $\Gamma, \Delta, \Lambda$ and $\Xi$ as in Theorem 3.1. Then we compute that
$$\begin{array}{c}
\Gamma=\left(
\begin{array}{cc}
0&1\\
0&-1
\end{array}
\right), \Delta=\left(
\begin{array}{cc}
-i&-i\\
0&0
\end{array}
\right),\\
\Lambda=\left(
\begin{array}{cc}
-i&-i\\
0&0
\end{array}
\right), \Xi=\left(
\begin{array}{cc}
1&1\\
0&0
\end{array}
\right).
\end{array}$$ \\
This completes the proof by Theorem 3.1.\end{proof}

\end{document}